\documentclass[12pt]{amsart}
\usepackage[margin=1in]{geometry}
\usepackage{hyperref}

\usepackage{amsthm, amsfonts, amsmath, amssymb, amsxtra}
									  
\numberwithin{equation}{section}
\usepackage{cleveref}
					  
\usepackage{enumitem}
\usepackage{mathtools}

\theoremstyle{definition}
\newtheorem{theorem}{Theorem}[section]
\newtheorem{corollary}[theorem]{Corollary}
\newtheorem{lemma}[theorem]{Lemma}

\newtheorem*{corollary*}{Corollary} 
\newtheorem*{definition*}{Definition}

\newtheorem*{remark*}{Remark}
\newtheorem{remark}[theorem]{Remark}
\newtheorem{proposition}[theorem]{Proposition}
\newtheorem*{proposition*}{Proposition}

\crefname{theorem}{Theorem}{Theorems}
\crefname{lemma}{Lemma}{Lemmas}
\crefname{proposition}{Proposition}{Propositions}
\crefname{definition}{Definition}{Definitions}

\newcommand{\norm}[1]{\ensuremath{\left\| {#1} \right\|}}

\usepackage{pgf,tikz,pgfplots}
\pgfplotsset{compat=1.15}
\usetikzlibrary{arrows}
\usepackage{caption}
				
\title{On the One-dimensional Singular Abreu Equations}
\author{Young Ho Kim}
\address{Department of Mathematics, Indiana University, Bloomington, IN 47405, USA}
\email{yk89@iu.edu}

\begin{document}
	\begin{abstract} 
		Singular fourth-order Abreu equations have been used 
		to approximate minimizers of convex functionals
		subject to a convexity constraint in dimensions higher than or equal to two.
		For Abreu type equations,
		they often exhibit different solvability phenomena 
		in dimension one and dimensions at least two.
		We prove the analogues of these results for the variational problem and 
		singular Abreu equations in dimension one,
		and use the approximation scheme to obtain a characterization of 
		limiting minimizers to the one-dimensional variational problem.
	\end{abstract}
	
	\maketitle

	\section{Introduction and the Statement of Main Results}
	In this note, we consider a class of singular fourth-order Abreu equations in dimension one.
	In dimensions higher than or equal to two,
	singular Abreu equations have been used by various authors 
	in the approximation of minimizers of convex functionals with a convexity constraint.
	We will briefly recall these results below.
	On the other hand,
	for Abreu type equations,
	they often exhibit different solvability phenomena  
	in dimension one and dimensions at least two.
	We prove the analogues in dimension one,
	and find a characterization of limiting minimizers to a one-dimension variational problem
	by using this approximation scheme.
	
	Suppose $\Omega$ and $\Omega_0$ are bounded, smooth, convex domains in $\mathbb{R}^n$
	with $\Omega_0\Subset \Omega$.
	Let $\varphi \in C^5(\overline{\Omega})$ be a given convex function,
	and $F=F(x,z,p):\mathbb{R}^n\times\mathbb{R}\times\mathbb{R}^n\rightarrow\mathbb{R}$
	be a smooth Lagrangian that is convex 
	in the variables $z\in\mathbb{R}$ and $p\in\mathbb{R}^n$.
	Consider the variational problem
	\begin{equation}\label{hdvar}
	\begin{aligned}
		\inf_{u\in \overline{S}[\varphi,\Omega_0]} \int_{\Omega_0} F(x, u(x), Du(x)) \, dx
	\end{aligned}
	\end{equation}
	over the competitors $u$ with a convexity constraint given by
	\begin{equation}\label{hdvarcon}
	\begin{aligned}
		\overline{S}[\varphi,\Omega_0]
		=\{ u:\Omega\rightarrow\mathbb{R} \text{ convex, }
		u=\varphi \text{ on } \Omega\setminus\Omega_0 \}
		\text{.}
	\end{aligned}
	\end{equation}
	Because of the convexity constraint, 
	variational problems of this type are not easy to handle, 
	especially in numerical schemes \cite{BCMO,M16}.
	When $n\geq 2$ and the Lagrangian $F=F(x,z)$ does not depend on the gradient variable $p$,
	Carlier and Radice \cite{CR} introduced an approximation scheme for
	minimizers of the problem (\ref{hdvar})--(\ref{hdvarcon}).
	Le \cite{Singular_Abreu} extended this result to cover the case when the Lagrangian $F$
	could be split into $F(x,z,p) = F^0(x,z) + F^1(x,p)$ with appropriate conditions on $F^0$ and $F^1$, 
	and this result was followed by many other works including those of
	Le \cite{Convex_Approx, Twisted_Harnack} and of Le-Zhou \cite{Abreu_HD}. 
	One example of a problem of the type (\ref{hdvar})--(\ref{hdvarcon}) 
	is the Rochet-Chon\'e model \cite{RC} for the monopolist problem. 
	For this problem, 
	the Lagrangian is given by $F(x,z,p) = (|p|^q/q - x\cdot p + z)\eta_0 (x)$,
	where $q\in(1,\infty)$ and $\eta_0$ is a nonnegative Lipschitz function.
	
	The scheme introduced by Carlier and Radice in \cite{CR} for the functional
	\begin{equation}
	\begin{aligned}
		J_0(v)=\int_{\Omega_0}
		F(x,v(x)) \, dx
	\end{aligned}
	\end{equation}
	is to use uniformly convex solutions, 
	for $\varepsilon>0$, to the second boundary value problem
	\begin{align}\label{hd:eleq}
	\left\{
	\begin{aligned}
	\varepsilon U_\varepsilon^{ij} D_{ij} w_\varepsilon &= f_\varepsilon 
		:= \dfrac{\partial F}{\partial z} (x, u_\varepsilon)\chi_{\Omega_0} + 
			\dfrac{1}{\varepsilon} (u_\varepsilon-\varphi)
			\chi_{\Omega\setminus\Omega_0} &&\text{ in } \Omega\text{,}\\
		w_\varepsilon &= ({\det D^2 u_\varepsilon})^{-1} 
			&&\text{ in }\Omega\text{,}\\
		u_\varepsilon &= \varphi\text{, }w_\varepsilon = \psi &&\text{ on } \partial\Omega\text{,}
	\end{aligned}
	\right.
	\end{align}
	where $U_\varepsilon^{ij}=(\det D^2 u_\varepsilon)(D^2u_\varepsilon)^{-1}$
	is the cofactor matrix of the Hessian matrix $D^2 u_{\varepsilon}$,
	in approximating minimizers of the variational problem (\ref{hdvar})--(\ref{hdvarcon}).
	Here $\chi_E$ denotes the characteristic function of the set $E$.
	The first two equations in (\ref{hd:eleq}) arise as critical points of the 
	approximate functional
	\begin{equation}
	\begin{aligned}
		J_0(v)
		+ \frac{1}{2\varepsilon}\int_{\Omega\setminus\Omega_0} (v-\varphi)^2 \, dx
		- \varepsilon \int_\Omega \log \det D^2 v \, dx
		\text{,}
	\end{aligned}
	\end{equation}
	and the boundary conditions correspond to the prescribed boundary values
	of the function $u_\varepsilon$ and its Hessian determinant $\det D^2 u_\varepsilon$.
	Due to these boundary conditions, (\ref{hd:eleq}) is called a second boundary value problem.
	In the more general case when $F(x,z,p) = F^0(x,z) + F^1(x,p)$,
	Le \cite{Singular_Abreu} uses the same approximation scheme with 
	$f_\varepsilon$ in (\ref{hd:eleq}) replaced by 
	\begin{equation}\label{singfe}
	\begin{aligned}
		f_\varepsilon = 
			\left\{ \dfrac{\partial F^0}{\partial z} (x, u_\varepsilon) 
			- \dfrac{\partial}{\partial x_i} 
			\left( \dfrac{\partial F^1}{\partial p_i} (x,Du_\varepsilon) \right)
			\right\} \chi_{\Omega_0} +
			\dfrac{1}{\varepsilon} (u_\varepsilon-\varphi) \chi_{\Omega\setminus\Omega_0}
			\text{.}
	\end{aligned}
	\end{equation}
	
	The first two equations
	\begin{equation}
	\begin{aligned}
		U_\varepsilon^{ij} D_{ij} w_\varepsilon =\varepsilon^{-1}f_\varepsilon
		\text{, }
		w_\varepsilon=(\det D^2 u_\varepsilon)^{-1}
	\end{aligned}
	\end{equation}
	form a fourth-order nonlinear equation of Abreu type \cite{Ab}
	that arises in the problem of finding K\"ahler metrics of constant scalar curvature
	for toric manifolds \cite{Donaldson}.
	The divergence term $\frac{\partial}{\partial x_i} 
	\left(\frac{\partial F^1}{\partial p_i} (x,Du_\varepsilon)\right)$
	added for $f_\varepsilon$ in the general case (\ref{singfe})
	is only guaranteed to be a measure when $u_\varepsilon$ is convex;
	hence (\ref{hd:eleq}) is called a \emph{singular Abreu equation}.

	We recall how the approximation was used in Carlier-Radice \cite{CR} and 
	Le \cite{Singular_Abreu, Convex_Approx}.
	First, an arbitrary uniformly convex solution to the equation (\ref{hd:eleq}) 
	(with $f_\varepsilon$ given by (\ref{singfe}) in \cite{Singular_Abreu, Convex_Approx})
	is shown to satisfy a priori $W^{4,s}$ estimates for all $s\in(n,\infty)$;
	then the Leray-Schauder degree theory and the a priori estimates
	yield the existence of solution to the equation.
	Next, it is proved that after extracting a subsequence $\varepsilon_k\rightarrow 0$,
	solutions $(u_{\varepsilon_k})_k$ are shown to converge uniformly
	on compact subsets of $\Omega$ 
	to a minimizer of the variational problem (\ref{hdvar})--(\ref{hdvarcon}).

	The previously mentioned results study the case when $n\geq 2$;
	we will focus on the one-dimensional case in problem (\ref{1deqn}) in this note,
	and it is not clear if similar results hold.
	The reason is as follows.
	The one-dimensional Abreu equation
	\begin{equation}\label{cweqn} 
	\begin{aligned} 
		(1/u'')'' = U^{ij} D_{ij} w = f
	\end{aligned}
	\end{equation}
	was studied by Chau and Weinkove in \cite[Proposition 3.2]{CW}
	in the case when the right-hand side $f$ is
	a function of only the spatial variable.
	For solutions for the second boundary value problem to (\ref{cweqn}) to exist,
	$f$ should satisfy a ``stability'' condition (see \cite[(3.2)]{CW});
	this is different from the case when $n\geq 2$,
	where the second boundary value problem for the Abreu equation
	has a solution if $f\in L^t(\Omega)$, $t>n$, as proved by Le in \cite{Le16}.
	
	In this note, problem (\ref{1deqn}) on the other hand involves a singular term.
	As we can see in Theorem \ref{mainthm}\ref{thm:est},
	``stability'' conditions are not required for solutions to this type of
	equations to exist.
	Contrary to the existence result for equations without singular terms,
	this result resembles the higher-dimensional counterpart.
	
	To formulate the one-dimensional problem,
	first note that $U_\varepsilon^{ij} D_{ij} w_\varepsilon = w_\varepsilon''$ when $n=1$.
	Without loss of generality, we can assume that $\Omega=(-1,1)$ and $\Omega_0 = (a,b)$, where $-1<a<b<1$.
	Then our second boundary value problem for the singular Abreu equation 
	in dimension one is given by
	\begin{equation}\label{1deqn} 
	\begin{aligned} 
	\left\{
		\begin{aligned} 
			\varepsilon w_\varepsilon'' &= f_\varepsilon 
				:= 
				\dfrac{1}{\varepsilon} (u_\varepsilon-\varphi) \chi_{(-1,1)\setminus(a,b)} \\
				&\qquad + \left(F_z^0(x,u_\varepsilon) - F_{px}^1(x,u_\varepsilon')
				-F_{pp}^1(x,u_\varepsilon')u_\varepsilon''\right)\chi_{(a,b)} &\text{ in }(-1,1)\text{,}\\
			w_\varepsilon &= 1/u_\varepsilon'' &\text{ in } (-1,1)\text{,}\\
			u_\varepsilon(\pm 1) &= 0\text{,} 
			\quad\text{and }
			w_\varepsilon(\pm 1) = \rho_\pm > 0 \text{.}
		\end{aligned}
	\right.
	\end{aligned}
	\end{equation}
	Here $\varphi$ is assumed to be smooth on $[-1,1]$, $\varphi(\pm 1) = 0$ 
	and satisfies $\varphi'' \geq c_0 > 0$.
	The first two equations of (\ref{1deqn}) arise as critical point of the functional
	\begin{equation}\label{jedef}
	\begin{aligned}
		J_\varepsilon (v):=
		\int_a^b F(x, v(x), v'(x)) \, dx - \varepsilon\int_{-1}^1 \log v''(x) \, dx
		+ \frac{1}{2\varepsilon}\int_{(-1,a)\cup (b,1)} (v-\varphi)^2 \, dx
		\text{,}
	\end{aligned}
	\end{equation}
	where the Lagrangian $F$ is given by 
	\begin{align}\label{lagdef}
		F(x,z,p) = F^0 (x,z) + F^1(x,p)
		\text{.}
	\end{align}
	We also assume that $F^0$ and $F^1$ satisfy
	\begin{itemize}
		\item[(F1)] $F^0,F^1\in C^2([-1,1]\times\mathbb{R})$,
		\item[(F2)] $F^0$ is convex in $z$,
		\item[(F3)] $F^1$ is convex in $p$ so that $F_{pp}^1 (x,p) \geq 0$,
		\item[(F4)] For smooth, increasing functions 
			$\eta, \eta_1: [0,\infty)\rightarrow[0,\infty)$
			and a positive constant $D_*$,	
			we have for all $x\in [-1,1]$ and $p,z\in\mathbb{R}$,
		\begin{equation}\label{Fcond}
		\begin{aligned}
			&|F^0 (x,z)| + |F_z^0(x,z)| \leq \eta(|z|)
			\text{,}\quad
			|F_{px}^1 (x,p)| \leq D_* (1 + |p|) 
			\text{,}\quad\text{and }\\
			&|F_p^1 (x,p) | + |F_{pp}^1(x,p)| \leq \eta_1 (|p|)
			\text{.}
		\end{aligned}
		\end{equation}
	\end{itemize}
	One example of a one-dimensional Lagrangian $F=F(x,z,p)$ satisfying (F1)--(F4) is 
	\begin{align*}
		F(x,z,p)=\left(\frac{p^2}{2} - px + z \right) \eta_0(x)
		\text{,}
	\end{align*}
	where $\eta_0$ is a nonnegative smooth function on $[-1,1]$.
	Here $F=F_0+ F_1$, where
	\begin{align*}
		F_0(x,z) = z\eta_0(x)
		\quad\text{and }
		F_1(x,p) = \left(\frac{p^2}{2} - px\right)\eta_0(x)
	\end{align*}
	are smooth, convex (in $z$ and $p$, respectively) functions
	whose derivatives satisfy the growth estimate in (\ref{Fcond}). 
	Since $\eta_0\geq 0$, (F1) and (F3) are also satisfied.

	Our main result is the following theorem.
	\begin{theorem}\label{mainthm}
		Let $-1<a<b<1$.
		Assume that $\varphi$ is a smooth, uniformly convex function on $[-1,1]$ 
		with $\varphi(\pm 1) = 0$ and $\varphi''\geq c_0 >0$.
		Assume the Lagrangian $F$ given by (\ref{lagdef}) satisfies (F1)--(F4) above.	
		Then the following hold.
		\begin{enumerate}[label=(\roman*)]
			\item There is a constant 
				$\varepsilon_0=\varepsilon_0(a,b,D_*,\rho_\pm,\varphi,\eta,\eta_1)\in (0,1)$
				such that for $\varepsilon\in(0,\varepsilon_0)$, 
				the problem (\ref{1deqn})
				has a uniformly convex $W^{4,\infty}(-1,1)$ solution $u_\varepsilon$.
				Furthermore, there is a constant 
				$\widetilde{C}=\widetilde{C}(a,b,D_*,\rho_\pm,\varphi,\eta,\eta_1)>0$
				such that
				\begin{equation}\label{west}
				\begin{aligned}
					u_\varepsilon'' \geq \widetilde{C}\varepsilon
					\quad\text{in } (a,b)
					\text{.}
				\end{aligned}
				\end{equation}
				\label{thm:est}
			\item Let $(u_\varepsilon)_{0<\varepsilon<1}$ be $W^{4,\infty}(-1,1)$ solutions to (\ref{1deqn}).
				Then, there is a sequence $\varepsilon_k \rightarrow 0$ such that $u_{\varepsilon_k}$
				converges uniformly on compact intervals in $(-1,1)$ to a convex function $u$ 
				in $(-1,1)$ that satisfies $u=\varphi$ outside $(a,b)$ 
				and minimizes the functional
				\begin{equation}\label{1dfnctl}
				\begin{aligned}
					J(v) = \int_a^b F(x, v(x), v'(x)) \, dx
				\end{aligned}
				\end{equation}
				over $v\in \overline{S}[\varphi]$,
				where $\overline{S}[\varphi]$ is given by
				\begin{equation}\label{1dvarprob}
				\begin{aligned}
					\overline{S}[\varphi]
					= \{ v :
						v\text{ is convex on }[-1,1]
						\text{ and }v=\varphi
						\text{ outside }(a,b) 
					\}
					\text{.}
				\end{aligned}
				\end{equation}
				\label{thm:conv}
			\item Let $q\in[1,\infty)$ be fixed,
				and assume that $u$ is given as in \ref{thm:conv}.
				If the Lagrangian $F$ also satisfies
				\begin{equation}\label{sec3growth}
				\begin{aligned}
					|F_{zz}^0 (x,z)| \leq \eta_2(|z|)
					\quad\text{in }[-1,1]\times\mathbb{R}
				\end{aligned}
				\end{equation}
				for a smooth, increasing function $\eta_2$, 
				then there is a function $w \in L^q (a,b)$ 
				which is a weak limit in $L^q(a,b)$ of a subsequence of 
				$(\varepsilon_k w_{\varepsilon_k})_k$, and satisfies
				\begin{equation}\label{eleqdist}
				\begin{aligned}
					w'' = F_z^0(x,u) - (F_p^1(x,u'))'
					\quad\text{in }(a,b)
				\end{aligned}
				\end{equation}
				in the sense of distributions.
				\label{thm:eqn}
		\end{enumerate}
	\end{theorem}

	\begin{remark}
		For Theorem \ref{mainthm}\ref{thm:est}--\ref{thm:conv},
		the proofs are similar to that of Le 
		\cite{Singular_Abreu, Convex_Approx, Twisted_Harnack}.
		Since $U_\varepsilon^{ij} D_{ij} w_\varepsilon$ is much simpler in the one-dimensional case
		(as it is just $(1/u_\varepsilon'')''$),
		we do not need to invoke regularity results used in the higher-dimensional case.
		Moreover, we obtain $W^{4,\infty}(-1,1)$ estimates in Theorem \ref{mainthm}\ref{thm:est} instead of the 
		$W^{4,s}$ estimates in higher dimensions.
		In Theorem \ref{mainthm}\ref{thm:conv} we need an additional step, 
		as part of the proofs in the higher-dimensional case 
		do not carry over to the one-dimensional case;
		see Remark \ref{rmk:gradconv}.
	\end{remark}
	
	\begin{remark}
		The estimate (\ref{west}) is new.
		It is not known if a similar estimate holds in higher dimensions
		for solutions to (\ref{hd:eleq}) with $f_\varepsilon$ given by (\ref{singfe}).
	\end{remark}
	
	\begin{remark}\label{rmk:Lions}\leavevmode
	\begin{enumerate}
		\item Theorem \ref{mainthm}\ref{thm:eqn} is related to the result of Lions \cite{Lions}.
		Suppose $\Omega_0\subset\mathbb{R}^n$ is an open, bounded, smooth and strongly convex domain.
		Then, Lions showed that the minimizer $u$ of the functional
		\begin{equation}
		\begin{aligned}
			\int_{\Omega_0} \left[\frac{1}{2} |Du|^2 -fu + f_i D_i u\right]  dx
		\end{aligned}
		\end{equation}
		over all convex functions $u\in H_0^1(\Omega_0)$ satisfies, 
		in the sense of distributions,
		\begin{equation}
		\begin{aligned}
			-\Delta u - f - D_i f_i = D_{ij} \mu_{ij}
		\end{aligned}
		\end{equation}
		where $(\mu_{ij})_{1\leq i,j\leq n}$ is a symmetric nonnegative matrix of Radon measures.
		Also see Carlier \cite{Car02}.

		The constraint for our variational problem (\ref{1dfnctl})--(\ref{1dvarprob}) 
		is that each competitor function
		has a convex extension that agrees with a given convex function $\varphi$
		outside $\Omega_0$.
		In addition to the Dirichlet boundary condition 
		$u=\varphi$ on $\partial\Omega_0$,
		this puts an additional restriction on the gradient of the minimizer 
		at the boundary $\partial\Omega_0$.
		Therefore, the result of Lions cannot be easily applied.

		We instead use the approximation scheme 
		in Theorem \ref{mainthm}\ref{thm:est}--\ref{thm:conv}
		to show that (\ref{eleqdist}) holds, 
		where $w$ is an $L^q$ function instead of being just a measure.

	\item As we use the approximation scheme 
		in Theorem \ref{mainthm}\ref{thm:est}--\ref{thm:conv},
		in Theorem \ref{mainthm}\ref{thm:eqn} we can only characterize 
		minimizers to (\ref{1dfnctl})--(\ref{1dvarprob}) 
		given by limits of solutions to (\ref{1deqn}).
		In certain cases (for instance, if the minimizer is unique),
		all solutions can be approximated,
		but this is not guaranteed in general.
		It would be interesting to know if there is a characterization
		for minimizers that are not limits of solutions to (\ref{1deqn}).
	\end{enumerate}
	\end{remark}
	
	The rest of the note is organized as follows.
	In Section \ref{sec:1dest}, we prove two estimates 
	satisfied by the solutions to (\ref{1deqn});
	one is the a priori estimate used to prove 
	the first part of Theorem \ref{mainthm}\ref{thm:est},
	the other is the estimate in (\ref{west}).
	This proves Theorem \ref{mainthm}\ref{thm:est}.
	In Section \ref{sec:1dconvthm}, we prove Theorem \ref{mainthm}\ref{thm:conv}
	and in Section \ref{sec:eleqthm}, we prove Theorem \ref{mainthm}\ref{thm:eqn}.
	The final section, Section \ref{sec:conclusion} contains summary of the note
	and some possible directions for future research.

	\section{Estimates and Existence of Solutions}\label{sec:1dest}
	In this section, we prove Theorem \ref{mainthm}\ref{thm:est}.
	The first statement can be proved using degree theory and 
	the a priori $W^{4,\infty}$ estimate in Proposition \ref{aprioriest} below.
	For this, we will mostly follow Le \cite[Section 2]{Twisted_Harnack},
	but since we are working with a simpler equation, some steps can be simplified.
	We will prove the second estimate (\ref{west}) in the process of proving the $W^{4,\infty}$ estimate. 

	In the following, we will always assume that $\varepsilon$ satisfies 
	$0<\varepsilon<\varepsilon_0<1$.	
	\begin{proposition}[A priori $W^{4.\infty}$ estimate]\label{aprioriest}
		Suppose $u_\varepsilon$ is a uniformly convex $W^{4,\infty}(-1,1)$ 
		solution to (\ref{1deqn}), 
		where the Lagrangian $F$ satisfies (F1)--(F4).
		If $0 < \varepsilon < \varepsilon_0$, where $\varepsilon_0$ is a small number
		depending only on $a,b,D_*,\rho_\pm,\varphi,\eta,\eta_1$, 
		then there is $C(\varepsilon) > 0$ such that
		\begin{equation}
		\begin{aligned}
			\norm{u_\varepsilon}_{W^{4,\infty}(-1,1)} \leq C(\varepsilon)
			\text{.}
		\end{aligned}
		\end{equation}
	\end{proposition}
	
	Throughout the section, $u_\varepsilon$ will denote 
	a uniformly convex $W^{4,\infty}(-1,1)$ solution to (\ref{1deqn}),
	and we will use numbered constants $C_n$ and $D_n$
	to denote positive constants that do not depend on the solution $u_\varepsilon$
	but only on $a$, $b$, $D_*$, the boundary values $\rho_\pm$, 
	and the functions $\varphi$, $\eta$, $\eta_1$.
	We will write $C_n$ and $D_n$ for constants that do not depend on $\varepsilon$,
	while for constants that depend on $\varepsilon$ the dependency will be explicitly stated.

	We start by getting an $L^\infty$ bound for $u_\varepsilon$.

	\begin{lemma} 
		If $\varepsilon < \varepsilon_0$ where 
		$\varepsilon_0=\varepsilon_0(a,b,D_*,\rho_\pm,\varphi,\eta,\eta_1)$ is small, then 
		\begin{equation}\label{1-ide8}
		\begin{aligned}
			\norm{u_\varepsilon}_{L^\infty(-1,1)} < C_3
			=C_3(a,b,D_*,\rho_\pm,\varphi,\eta,\eta_1)
			\text{.}
		\end{aligned}
		\end{equation}
	\end{lemma}
	
	\begin{proof}
		If $\psi$ is a $C^2$ function on $[-1,1]$ satisfying $\psi(\pm 1) = 0$, 
		then we can multiply the first equation in (\ref{1deqn}) 
		by $\psi$ and integrate by parts to get
		\begin{equation*} 
		\begin{aligned} 
			\int_{-1}^1 f_\varepsilon \psi \, dx
			= \varepsilon \int_{-1}^1 w_\varepsilon '' \psi\, dx 
			= \varepsilon \left([w_\varepsilon' \psi]_{-1}^1 - \int_{-1}^1 w_\varepsilon ' \psi'\,dx \right) 
			= -\varepsilon\int_{-1}^1 w_\varepsilon ' \psi' \, dx
			\text{.}
		\end{aligned}
		\end{equation*}
		Dividing by $\varepsilon$ and integrating by parts again gives
		\begin{equation}\label{1-ide1} 
		\begin{aligned} 
			\frac{1}{\varepsilon}\int_{-1}^1 f_\varepsilon \psi \,dx
			&= -[w_\varepsilon \psi']_{-1}^1 + \int_{-1}^1 w_\varepsilon \psi''\,dx 
			\text{.}
		\end{aligned}
		\end{equation}

		Setting $\psi = u_\varepsilon - \varphi$ in (\ref{1-ide1}) 
		and substituting $f_\varepsilon$ from (\ref{1deqn}), 
		we find that the left-hand side of (\ref{1-ide1}) becomes
		\begin{equation}\label{1-ide1-1}
		\begin{aligned} 
			\frac{1}{\varepsilon}\int_{-1}^1 f_\varepsilon\psi\,dx
			&= \frac{1}{\varepsilon}\int_a^b f_\varepsilon(u_\varepsilon-\varphi)\,dx
			+ \frac{1}{\varepsilon^2}\int_{(-1,a)\cup(b,1)} (u_\varepsilon-\varphi)^2\,dx
			\text{,}
		\end{aligned}
		\end{equation}
		where
		\begin{equation}\label{1dfeexpansion}
		\begin{aligned}
			\frac{1}{\varepsilon}\int_a^b f_\varepsilon(u_\varepsilon-\varphi)\,dx
			&= \frac{1}{\varepsilon}\int_a^b F_z^0(x,u_\varepsilon) (u_\varepsilon-\varphi)\,dx\\
			&- \frac{1}{\varepsilon}\int_a^b F_{px}^1(x,u_\varepsilon')(u_\varepsilon-\varphi)\,dx
			- \frac{1}{\varepsilon}\int_a^b F_{pp}^1(x,u_\varepsilon')u_\varepsilon''(u_\varepsilon-\varphi)\,dx
			\text{.}
		\end{aligned}
		\end{equation}
		
		For $\psi = u_\varepsilon -\varphi$, the right-hand side of (\ref{1-ide1}) becomes
		\begin{equation} 
		\begin{aligned} 
			-[w_\varepsilon \psi']_{-1}^1 &+ \int_{-1}^1 w_\varepsilon \psi''\,dx \\
			&=-\rho_{+} u_\varepsilon'(1) + \rho_{-}u_\varepsilon'(-1) + \rho_{+}\varphi'(1) - \rho_{-}\varphi'(-1)
			+ \int_{-1}^1 w_\varepsilon (u_\varepsilon '' - \varphi'')\,dx
			\text{.}
		\end{aligned}
		\end{equation}
		Since $w_\varepsilon = 1/u_\varepsilon''$, we have
		\begin{equation}\label{1-ide1-2}
		\begin{aligned} 
			\int_{-1}^1 w_\varepsilon (u_\varepsilon '' - \varphi'')\,dx
			= \int_{-1}^1 1 - \frac{\varphi''}{u_\varepsilon''}\,dx
			= 2 - \int_{-1}^1 \frac{\varphi''}{u_\varepsilon''}\,dx
			\text{.}
		\end{aligned}
		\end{equation}
		Because $u_\varepsilon <0$ in $(-1,1)$ and $u_\varepsilon(\pm 1) = 0$,
		we get $u_\varepsilon'(1)>0>u_\varepsilon'(-1)$.
		Therefore, as $\rho_\pm > 0$, 
		$-\rho_{+} u_\varepsilon'(1) + \rho_{-}u_\varepsilon'(-1) < 0$.
		Using (\ref{1-ide1-1})--(\ref{1-ide1-2}), we rewrite (\ref{1-ide1}) as 
		\begin{equation}\label{1-ide2}
		\begin{aligned}
			\int_{-1}^1 \frac{\varphi''}{u_\varepsilon''}\,dx
			&+ \frac{1}{\varepsilon^2} \int_{(-1,a)\cup (b,1)} (u_\varepsilon-\varphi)^2\,dx
			+ \frac{1}{\varepsilon} \int_a^b f_\varepsilon (u_\varepsilon-\varphi)\,dx\\
			&= -\rho_{+} u_\varepsilon'(1) + \rho_{-}u_\varepsilon'(-1) + \rho_{+}\varphi'(1) - \rho_{-}\varphi'(-1) + 2 \\
			&< \rho_{+}\varphi'(1) - \rho_{-}\varphi'(-1) + 2 =:C_1
			\text{.}
		\end{aligned}
		\end{equation}

		Now, we consider the following cases as in Le-Zhou \cite[pp.27--28]{Abreu_HD}.

		\emph{Case 1.} $u_\varepsilon(x)\geq \varphi(x)$ for some $x\in(a,b)$.
		Then, as $u_\varepsilon$ is a negative convex function 
		with $u_\varepsilon(-1) = 0$, we have
		\begin{equation*}
		\begin{aligned}
			|u_\varepsilon(y)|
			\leq \frac{y+1}{x+1} |u_\varepsilon(x)|
			\leq \frac{2}{1+a} \norm{\varphi}_{L^\infty(-1,1)}
			\quad\text{for }
			y\in(x,1)
			\text{.}
		\end{aligned}
		\end{equation*}
		We can also get a similar bound when $y\in (-1,x)$.
		Putting these together, we conclude that the $L^\infty$ norm of $u_\varepsilon$ is bounded
		independent of $\varepsilon$, as desired.

		\emph{Case 2.} $u_\varepsilon \leq \varphi$ in $(a,b)$.
		First, we note that as $F_{pp}^1\geq 0$, $u_\varepsilon\leq\varphi$
		and $u_\varepsilon'' >0$,
		\begin{equation}\label{1-ide2-0}
		\begin{aligned}
			\frac{1}{\varepsilon}\int_a^b F_{pp}^1 (x,u_\varepsilon')u_\varepsilon''
			(u_\varepsilon-\varphi)\,dx \leq 0
			\text{.}
		\end{aligned}
		\end{equation}

		Next, by the convexity of $F^0$ and (\ref{Fcond}), we have
		\begin{equation}\label{1-ide2-2}
		\begin{aligned}
			-\frac{1}{\varepsilon}\int_a^b F_z^0(x,u_\varepsilon)(u_\varepsilon-\varphi)\,dx
			&\leq -\frac{1}{\varepsilon}\int_a^b F_z^0(x,\varphi)(u_\varepsilon-\varphi)\,dx\\
			&\leq \frac{b-a}{\varepsilon} \eta(\norm{\varphi}_{L^\infty(-1,1)})
			(\norm{u_\varepsilon}_{L^\infty(-1,1)} + \norm{\varphi}_{L^\infty(-1,1)})
			\text{.}
		\end{aligned}
		\end{equation}
		Because $u_\varepsilon$ is convex and $u_\varepsilon(\pm 1) = 0$, 
		for any interval $(t_1, t_2)$ contained in $(-1,1)$
		we have the gradient bound
		\begin{equation}\label{1-ide2-1}
		\begin{aligned}
			|u_\varepsilon'(x)|
			&\leq \frac{|u_\varepsilon(x)|}{\min(x-(-1), 1-x)}
			&\leq \frac{\norm{u_\varepsilon}_{L^\infty(-1,1)}}{\min(t_1+1,1-t_2)}
			\quad\text{for }x\in(t_1,t_2)
			\text{.}
		\end{aligned}
		\end{equation}
		Finally, 
		from (\ref{1-ide2-1}) (with $t_1 = a$ and $t_2 = b$) and (\ref{Fcond}), we have
		\begin{equation}\label{1-ide2-3}
		\begin{aligned}
			&\frac{1}{\varepsilon}\int_a^b F_{px}^1(x,u_\varepsilon')(u_\varepsilon-\varphi)\,dx \\
			&\leq \frac{b-a}{\varepsilon}D_*(1+\norm{u_\varepsilon'}_{L^\infty(a,b)})
			(\norm{u_\varepsilon}_{L^\infty(-1,1)} + \norm{\varphi}_{L^\infty(-1,1)})\\
			&\leq \frac{1}{\varepsilon} (b-a)D_*
			\left(1 + \frac{\norm{u_\varepsilon}_{L^\infty(-1,1)}}{\min(a+1,1-b)} \right)
			(\norm{u_\varepsilon}_{L^\infty(-1,1)} + \norm{\varphi}_{L^\infty(-1,1)})
			\text{.}
		\end{aligned}
		\end{equation}

		Putting (\ref{1-ide2-0}), (\ref{1-ide2-2}) and (\ref{1-ide2-3}) 
		together with (\ref{1-ide2}) and (\ref{1dfeexpansion}) yields
		\begin{equation}\label{1-ide3}
		\begin{aligned}
			c_0 \int_{-1}^1 \frac{1}{u_\varepsilon''}\,dx
			+ \frac{1}{\varepsilon^2}\int_{(-1,a)\cup(b,1)}(u_\varepsilon-\varphi)^2\,dx
			&< C_1 - \frac{1}{\varepsilon}\int_a^b f_\varepsilon(u_\varepsilon-\varphi)\,dx \\
			&\leq\frac{C_2}{\varepsilon}(\norm{u_\varepsilon}_{L^\infty(-1,1)}^2 + 1)
			\text{.}
		\end{aligned}
		\end{equation}
		Here, we used the assumption that $\varepsilon<\varepsilon_0<1$ to absorb the $C_1$
		term into $\frac{C_2}{\varepsilon}$.
		Thus, $C_2$ will depend on $\varepsilon_0$. 
		However, as $\varepsilon_0$ depends on the 
		same set of variables $a,b,D_*,\rho_\pm,\varphi,\eta,\eta_1$ 
		as the constants $C_n$ do (stated at the beginning of the section),
		we can still denote the constant by $C_2$. 

		Now, we are ready to obtain the uniform $L^\infty$ bound for $u_\varepsilon$.
		Suppose that $u_\varepsilon$ attains its minimum on $t\in (-1,1)$, so that
		\begin{align*}
			|u_\varepsilon(t)| = \norm{u_\varepsilon}_{L^\infty(-1,1)}
			\text{.}
		\end{align*}
		Because $-1<a<b<1$, we either have $t<b$ or $t>a$.
		If $t<b$, as $u_\varepsilon$ is a negative convex function with $u_\varepsilon(1)=0$,
		\begin{equation*}
		\begin{aligned}
			|u_\varepsilon(x)|
			&\geq \frac{1-x}{1-t} |u_\varepsilon(t)| \\
			&\geq \frac{1-x}{2}\norm{u_\varepsilon}_{L^\infty(-1,1)} 
			\quad\text{in } (b,1)
			\text{.}
		\end{aligned}
		\end{equation*}
		Therefore, we have
		\begin{equation}\label{1-ide3-1}
		\begin{aligned}
			\int_b^1 (u_\varepsilon-\varphi)^2\,dx
			&\geq \frac{1}{2} \int_b^1 u_\varepsilon^2 - 2\varphi^2\,dx\\
			&\geq \frac{1}{2}
			\left( \norm{u_\varepsilon}_{L^\infty(-1,1)}^2 \int_b^1 \left(\frac{1-x}{2}\right)^2 \, dx
				- 2(1-b)\norm{\varphi}_{L^\infty(-1,1)}^2 \right)\\
			&= \frac{1}{2}
			\left( \frac{(1-b)^3}{12} \norm{u_\varepsilon}_{L^\infty(-1,1)}^2 
				- 2(1-b)\norm{\varphi}_{L^\infty(-1,1)}^2 \right)
			\text{.}
		\end{aligned}
		\end{equation}
		On the other hand, suppose $ t>a$. Following the same argument, we obtain 
		\begin{equation}\label{1-ide3-1'}
		\begin{aligned}
			\int_{-1}^a (u_\varepsilon-\varphi)^2\,dx
			\geq \frac{1}{2}
			\left( \frac{(a+1)^3}{12} \norm{u_\varepsilon}_{L^\infty(-1,1)}^2 
				- 2(a+1)\norm{\varphi}_{L^\infty(-1,1)}^2 \right)
			\text{.}
		\end{aligned}
		\end{equation}
		Hence, if $\varepsilon$ is small enough, 
		then combining (\ref{1-ide3}) with (\ref{1-ide3-1}) when $t<b$ 
		(or (\ref{1-ide3-1'}) if $t>a$) and $\int_{-1}^1 \frac{1}{u_\varepsilon''}\, dx > 0$,
		we obtain the $L^\infty$ bound of $u_\varepsilon$ on $(-1,1)$ independent of $\varepsilon$.
	\end{proof}

	Now, we use the gradient bound (\ref{1-ide2-1}) with $t_1 = a$ and $t_2 = b$. 
	Combining it with the $L^\infty$ bound (\ref{1-ide8}), we get the following estimate.
	\begin{corollary} 
		If $x \in (a,b)$ and $\varepsilon < \varepsilon_0$ for 
		$\varepsilon_0=\varepsilon_0(a,b,D_*,\rho_\pm,\varphi,\eta,\eta_1)$ small, 
		then we have
		\begin{equation}\label{1dgrad}
		\begin{aligned}
			|u_\varepsilon'(x)|
			&\leq \frac{C_3}{\min(a+1, 1-b)}=:D_1
			\text{.}
		\end{aligned}
		\end{equation}

		From (\ref{1-ide3}) and (\ref{1-ide8}), we also have
		\begin{equation}\label{1-ide6}
		\begin{aligned}
			c_0 \int_{-1}^1 \frac{1}{u_\varepsilon''}\,dx
			+ \frac{1}{\varepsilon^2}\int_{(-1,a)\cup(b,1)}(u_\varepsilon-\varphi)^2\,dx
			\leq \frac{C_4}{\varepsilon}
			\text{,}
			\quad\text{where }C_4:= C_2 (C_3^2 + 1)
			\text{.}
		\end{aligned}
		\end{equation}
	\end{corollary}

	Next, we show a lower bound for $w_\varepsilon$
	(or equivalently, an upper bound for $u_\varepsilon''$).

	\begin{lemma} 
		If $\varepsilon < \varepsilon_0$ where 
		$\varepsilon_0=\varepsilon_0(a,b,D_*,\rho_\pm,\varphi,\eta,\eta_1)$ is small, then 
		\begin{equation}\label{1-ide9}
		\begin{aligned}
			w_\varepsilon(x) \geq C_5(\varepsilon)
			\text{,}\quad\text{thus } u_\varepsilon''(x) \leq C_5^{-1}(\varepsilon)
			\quad\text{if }x\in (-1,1)\text{.}
		\end{aligned}
		\end{equation}
	\end{lemma}

	\begin{proof}
		From the $L^\infty$ bound (\ref{1-ide8}), we have
		\begin{equation}\label{1-ide9-1}
		\begin{aligned}
			|f_\varepsilon|
			\leq \frac{1}{\varepsilon}
			\left(\norm{u_\varepsilon}_{L^\infty(-1,1)}
				+\norm{\varphi}_{L^\infty(-1,1)}\right)
			\leq \frac{1}{\varepsilon}
			\left(C_3 +\norm{\varphi}_{L^\infty(-1,1)}\right)
			\quad\text{outside }(a,b)
			\text{.}
		\end{aligned}
		\end{equation}
		In $(a,b)$, we have
		\begin{equation*}
		\begin{aligned}
			f_\varepsilon 
			&= F_z^0(x,u_\varepsilon) - F_{px}^1(x,u_\varepsilon') - F_{pp}^1 (x,u_\varepsilon')u_\varepsilon'' 
			&&\leq F_z^0(x,u_\varepsilon) - F_{px}^1(x,u_\varepsilon') \\
			&\leq \eta (\norm{u_\varepsilon}_{L^\infty(-1,1)}) 
			+ D_*(1 + \norm{u_\varepsilon'}_{L^\infty(a,b)}) 
			&&\leq \eta(C_3) + D_*(1 + D_1)
			\text{.}
		\end{aligned}
		\end{equation*}
		Therefore, setting
		\begin{equation*}
		\begin{aligned}
			M = M(\varepsilon) := 
			\frac{1}{\varepsilon}
			\max \left\{ 
				\frac{1}{\varepsilon} (C_3 + \norm{\varphi}_{L^\infty(-1,1)}),\, \eta(C_3) + D_*(1+D_1)
			\right\}
			\text{,}
		\end{aligned}
		\end{equation*}
		we get
		\begin{equation*}
		\begin{aligned}
			M \geq \frac{1}{\varepsilon} f_\varepsilon = w_\varepsilon''
			\text{.}
		\end{aligned}
		\end{equation*}
		
		Hence,
		\begin{equation*}
		\begin{aligned}
			v=\log w_\varepsilon - Mu_\varepsilon 
		\end{aligned} 
		\end{equation*}
		satisfies
		\begin{equation}\label{vv}
		\begin{aligned}
			v'' = \frac{w_\varepsilon'' - M}{w_\varepsilon}
			-\left(\frac{w_\varepsilon'}{w_\varepsilon}\right)^2
			\leq 0
			\text{.}
		\end{aligned}
		\end{equation}
		As the boundary values for $v$ are
		\begin{equation*}
		\begin{aligned}
			v(\pm 1)
			= \log w_\varepsilon(\pm 1) - Mu_\varepsilon(\pm 1) 
			= \log \rho_\pm
			\text{,}
		\end{aligned}
		\end{equation*}
		(\ref{vv}) implies that
		\begin{equation*}
		\begin{aligned}
			v(x) \geq \min\{v(1), v(-1)\}
			= \min\{\log \rho_+, \log \rho_-\} 
			\text{.}
		\end{aligned}
		\end{equation*}
		As a result,
		\begin{equation*}
		\begin{aligned}
			\log w_\varepsilon(x)
			= v(x) + Mu_\varepsilon(x)
			\geq \min\{\log \rho_+, \log \rho_-\}  - M(\varepsilon)C_3
			\text{,}
		\end{aligned}
		\end{equation*}
		which completes the proof of (\ref{1-ide9}) 
		for $C_5 (\varepsilon) := e^{\min\{\log \rho_+, \log \rho_-\} - M(\varepsilon) C_3}$.
	\end{proof}	

	Now we prove the following lemma, which implies the estimate (\ref{west}).
	\begin{lemma}\label{webddlemma}
		There is a constant $D_3=D_3(a,b,D_*,\rho_\pm,\varphi,\eta,\eta_1)>0$ 
		independent of $\varepsilon$ such that 
		if $\varepsilon<\varepsilon_0$ where
		$\varepsilon_0=\varepsilon_0(a,b,D_*,\rho_\pm,\varphi,\eta,\eta_1)$
		is small, we have
		\begin{equation}\label{lemmaineq}
		\begin{aligned}
			w_\varepsilon \leq \frac{D_3}{\varepsilon}
			\quad\text{in } (a,b)
			\text{.}	
		\end{aligned}
		\end{equation}	
	\end{lemma}
	
	\begin{proof}
		From (\ref{1deqn}), we have
		\begin{equation}
		\begin{aligned}
			(\varepsilon w_\varepsilon' + F_p^1 (x, u_\varepsilon'))'
			=\varepsilon w_\varepsilon'' + F_{px}^1(x,u_\varepsilon') + F_{pp}^1(x,u_\varepsilon')u_\varepsilon''
			= F_z^0 (x,u_\varepsilon)
			\quad\text{in }(a,b)
			\text{.}
		\end{aligned}
		\end{equation}
		Let us define
		\begin{equation}
		\begin{aligned}
			\lambda := \sup_{x\in (a,b)} 
			\left( \varepsilon w_\varepsilon'(x) + F_p^1(x,u_\varepsilon'(x)) \right)
			\text{.}
		\end{aligned}
		\end{equation}
		From (\ref{Fcond}) and (\ref{1-ide8}), $|F_z^0(x, u_\varepsilon)|\leq \eta(C_3)$.
		Therefore, we have
		\begin{equation*}
		\begin{aligned}
			\varepsilon w_\varepsilon'(x) + F_p^1(x, u_\varepsilon'(x)) \geq \lambda - \eta(C_3)(b-a)
			\quad\text{for }x \in (a,b)
			\text{.}
		\end{aligned}
		\end{equation*}
		We also have,
		from (\ref{Fcond}) and (\ref{1dgrad}),
		\begin{equation*}
		\begin{aligned}
			|F_p^1(x,u_\varepsilon')| \leq \eta_1(D_1)
			\quad\text{for }x\in(a,b)
			\text{.}
		\end{aligned}
		\end{equation*}
		Therefore, for all $x\in(a,b)$, we have
		\begin{equation}
		\begin{aligned}
			\varepsilon w_\varepsilon'(x)
			\geq \lambda-\eta(C_3)(b-a)-\eta_1(D_1) 
			=: \lambda - C_6
			\text{.}
		\end{aligned}
		\end{equation}

		Now, as $w_\varepsilon = 1/u_\varepsilon'' > 0$, (\ref{1-ide6}) gives us
		\begin{equation}\label{weintest}
		\begin{aligned}
			\frac{C_4}{c_0}
			&\geq \int_{-1}^1 \frac{\varepsilon}{u_\varepsilon''}\,dx
			\geq \int_a^b \varepsilon w_\varepsilon (x) \, dx \\
			&\geq \int_a^b \int_a^x \varepsilon w_\varepsilon'(t) \, dt \, dx \\
			&\geq \int_a^b \int_a^x \left(\lambda-C_6\right) \, dt \, dx \\
			&= \frac{(b-a)^2}{2} (\lambda - C_6)
			\text{.}
		\end{aligned}
		\end{equation}
		Therefore, we have
		\begin{equation}
		\begin{aligned}
			\lambda \leq \frac{2}{(b-a)^2} \frac{C_4}{c_0} + C_6 =: C_7
			\text{.}
		\end{aligned}
		\end{equation}
		This implies the estimate
		\begin{equation}\label{1-c}
		\begin{aligned}
			w_\varepsilon' (x)
			&\leq \frac{1}{\varepsilon} (\lambda + \eta_1(|u_\varepsilon'(x)|) ) 
			\leq \frac{1}{\varepsilon}(\lambda + \eta_1(D_1)) \\
			&\leq \frac{1}{\varepsilon}\left( C_7 + \eta_1(D_1) \right)
			\text{.}
		\end{aligned}
		\end{equation}

		Repeating the argument for 
		$\inf_{x\in (a,b)} \left( \varepsilon w_\varepsilon'(x) + F_p^1(x,u_\varepsilon'(x)) \right)$,
		we get 
		\begin{equation}\label{1-c1}
		\begin{aligned}
			w_\varepsilon' (x) \geq -\frac{1}{\varepsilon}\left( C_7 + \eta_1(D_1) \right)
			\text{.}
		\end{aligned}
		\end{equation}
		Hence, from (\ref{1-c}) and (\ref{1-c1}), for $x\in(a,b)$, we have
		\begin{equation}
		\begin{aligned}
			|w_\varepsilon'(x)| < \frac{D_2}{\varepsilon}
			\text{,}\quad\text{where }
			D_2 = \eta_1(D_1) + C_7
			\text{.}
		\end{aligned}
		\end{equation}
		This gives $|w_\varepsilon(x) - w_\varepsilon(y)| \leq \frac{(b-a)D_2}{\varepsilon}$
		for $x,y\in (a,b)$, and thus from (\ref{weintest})
		\begin{equation*}
		\begin{aligned}
			\frac{C_4}{c_0\varepsilon} 
			&\geq \int_a^b w_\varepsilon(y) \,dy 
			\geq (b-a)w_\varepsilon(x) - \int_a^b |w_\varepsilon(x)-w_\varepsilon(y)| \,dy \\
			&\geq (b-a)w_\varepsilon (x) - \frac{(b-a)^2 D_2}{\varepsilon}
			\text{.}
		\end{aligned}
		\end{equation*}
		Using this, we establish (\ref{lemmaineq})
		for $D_3 = (b-a)D_2 + \frac{C_4}{c_0(b-a)}$.
		This completes the proof of the lemma.
	\end{proof}

	Now we can prove the desired a priori estimate in Proposition \ref{aprioriest}.

	\begin{proof}[Proof of Proposition \ref{aprioriest}]
		From (\ref{1dgrad}) and (\ref{1-ide9}),
		we easily obtain
		\begin{equation}
		\begin{aligned}
			\norm{u_\varepsilon'}_{L^\infty(-1,1)}\leq D_1 + 2C_5^{-1}(\varepsilon)
			\text{.}
		\end{aligned}
		\end{equation}
		If $x \in (a,b)$, from (\ref{Fcond}) and the bounds on $u_\varepsilon$, $u_\varepsilon'$ and $u_\varepsilon''$ we have
		\begin{equation*}
		\begin{aligned}
			|f_\varepsilon(x)|
			&\leq |F_z^0(x,u_\varepsilon)| + |F_{px}^1(x,u_\varepsilon')| 
			+ |F_{pp}^1(x,u_\varepsilon')| \norm{u_\varepsilon''}_{L^\infty(-1,1)} \\
			&\leq \eta(C_3) + D_* (1 + D_1) + \eta_1(D_1) C_5^{-1}(\varepsilon)
			\text{.}
		\end{aligned}
		\end{equation*}
		Combining this with (\ref{1-ide9-1}) yields		
		\begin{equation}\label{1-ide10}
		\begin{aligned}
			|w_\varepsilon''(x)|
			= \frac{1}{\varepsilon}|f_\varepsilon(x)|
			&\leq C_8(\varepsilon)
			\quad\text{for }x\in(-1,1)\text{.}
		\end{aligned}
		\end{equation}
		This implies that 
		\begin{equation*}
		\begin{aligned}
			|w_\varepsilon'(x)-w_\varepsilon'(y)| 
			\leq C_8(\varepsilon) |x-y| \leq 2C_8(\varepsilon)
			\quad\text{for } x,y\in [-1,1]
			\text{.}
		\end{aligned}
		\end{equation*}

		As $w_\varepsilon (\pm 1)=\rho_\pm$, for $x\in[-1,1]$ we have
		\begin{equation*}
		\begin{aligned}
			|\rho_+ - \rho_-|
			=  \left| \int_{-1}^1 w_\varepsilon'(y) \,dy \right| 
			&\geq 2|w_\varepsilon'(x)| - \int_{-1}^1 |w_\varepsilon'(y)-w_\varepsilon'(x)| \,dy \\
			&\geq 2|w_\varepsilon'(x)| - 4C_8(\varepsilon)
			\text{.}
		\end{aligned}
		\end{equation*}
		Therefore, we have	
		\begin{equation}
		\begin{aligned}
			\norm{w_\varepsilon'}_{L_\infty(-1,1)} 
			\leq 2C_8(\varepsilon) + \frac{1}{2} |\rho_+ - \rho_-| =: C_9(\varepsilon)
			\text{.}
		\end{aligned}
		\end{equation}
		Combining this with
		$w_\varepsilon' = -{u_\varepsilon^{(3)}}/{(u_\varepsilon'')^2}$
		and (\ref{1-ide9}) yields
		\begin{equation}\label{1-ide11}
		\begin{aligned}
			\norm{u_\varepsilon^{(3)}}_{L^\infty(-1,1)} 
			\leq \norm{w_\varepsilon'}_{L^\infty(-1,1)} \norm{u_\varepsilon''}_{L^\infty(-1,1)}^2 
			\leq C_9(\varepsilon) C_5^{-2}(\varepsilon)
			=: C_{10}(\varepsilon)\text{.}
		\end{aligned}
		\end{equation}

		Similarly, 
		expanding $w_\varepsilon'' = (- {u_\varepsilon^{(3)}}/{(u_\varepsilon'')^2})'$
		and combining it
		with the previous estimates (\ref{1-ide9}), (\ref{lemmaineq}), (\ref{1-ide10}), (\ref{1-ide11})
		on $u_\varepsilon''$, $w_\varepsilon''$ and $u_\varepsilon^{(3)}$, we get
		\begin{equation}
		\begin{aligned}
			\norm{u_\varepsilon^{(4)}}_{L^\infty(-1,1)} \leq C_{11}(\varepsilon)
			\text{.}
		\end{aligned}
		\end{equation}
		We have obtained a priori bounds for $u_\varepsilon$
		and all of its derivatives up to the fourth-order.
		The proof of the proposition is complete.
	\end{proof}

	Finally, we prove Theorem \ref{mainthm}\ref{thm:est}.
	
	\begin{proof}[Proof of Theorem \ref{mainthm}\ref{thm:est}]
		The first part, the existence of uniformly convex 
		$W^{4,\infty}(-1,1)$ solutions to (\ref{1deqn}),
		follows from the a priori estimate in Proposition \ref{aprioriest}
		by using the Leray-Schauder degree theory as in Le \cite[pp.2275--2276]{Singular_Abreu}.
		The second part, the estimate (\ref{west}), follows from Lemma \ref{webddlemma}
		as $u_\varepsilon'' = w_\varepsilon^{-1}$.
	\end{proof}

	\section{Convergence of Solutions to a Minimizer}\label{sec:1dconvthm}
	In this section, we prove Theorem \ref{mainthm}\ref{thm:conv} on
	the convergence of solutions for (\ref{1deqn}) to a minimizer of 
	the variational problem (\ref{1dfnctl})--(\ref{1dvarprob}).
	We will mostly follow Le \cite{Singular_Abreu,Convex_Approx}.
	The main difference is the following lemma,
	which gives refined asymptotic behaviors of $u_\varepsilon'$ at $\pm 1$.
	An analogous result is not necessary in the higher-dimensional case;
	a weaker result is sufficient.
	(See Remark \ref{rmk:gradconv} for a detailed comparison.)

	\begin{lemma}\label{gradconvlemma}
		If $(u_\varepsilon)_{\varepsilon>0}$ are 
		$W^{4,\infty}(-1,1)$ solutions to (\ref{1deqn}),
		then we have 
		\begin{equation}\label{gradconv}
		\begin{aligned}
			\varepsilon u_\varepsilon'(\pm 1)\rightarrow 0
			\quad\text{as }\varepsilon\rightarrow 0
			\text{.}
		\end{aligned}
		\end{equation}
	\end{lemma}

	\begin{proof}
		It suffices to show by contradiction that
		$\varepsilon u_\varepsilon'(1) \rightarrow 0$ as $\varepsilon\rightarrow 0$.
		The same argument can be used to show 
		$\varepsilon u_\varepsilon'(-1)$ converges to 0 as $\varepsilon\rightarrow 0$,
		from which the Lemma follows.

		Assume, on the contrary, that there are $m>0$ 
		and a sequence $\varepsilon_n\rightarrow 0$ such that 
		\begin{equation}\label{raahyp}
		\begin{aligned}
			u_{\varepsilon_n}'(1) > \frac{m}{\varepsilon_n}
			\text{.}
		\end{aligned}
		\end{equation}
		First, by (\ref{1-ide6}) and the Cauchy–Schwarz inequality, if $b\leq x<y\leq 1$,
		\begin{equation*}
		\begin{aligned}
			|w_\varepsilon'(x) - w_\varepsilon'(y) |
			= \left|\int_x^y w_\varepsilon''(t)\, dt \right| 
			&\leq\left(\int_x^y 1\, dt \right)^{1/2}
			\left(\int_x^y w_\varepsilon''(t)^2\, dt \right)^{1/2} \\
			&\leq (y-x)^{1/2} 
			\left(\int_b^1\left[\frac{1}{\varepsilon^2}(u_\varepsilon-\varphi)\right]^2dt\right)^{1/2}\\
			&\leq C_4^{1/2} \varepsilon^{-3/2} |x-y|^{1/2}
			\text{.}
		\end{aligned}
		\end{equation*}
		Therefore, if $x,y\in (1-\varepsilon, 1)$, then 
		$|w_\varepsilon'(x) - w_\varepsilon'(y)|\leq C_4^{1/2}\varepsilon^{-1}$.
		Recalling that $w_\varepsilon = 1/u_\varepsilon'' >0 $, 
		we have
		\begin{equation*}
		\begin{aligned}
			\rho_+ 
			&> w_\varepsilon (1) - w_\varepsilon (1-\varepsilon) 
			= \int_{1-\varepsilon}^1 w_\varepsilon'(y) \,dy \\
			&\geq \int_{1-\varepsilon}^1 w_\varepsilon'(x) \,dy 
			- \int_{1-\varepsilon}^1 |w_\varepsilon'(y) - w_\varepsilon'(x)| \,dy \\
			&\geq\varepsilon w_\varepsilon'(x) - C_4^{1/2}
			\quad\text{for }x\geq 1-\varepsilon
			\text{.}
		\end{aligned}
		\end{equation*}
		This yields
		\begin{equation*}
		\begin{aligned}
			w_\varepsilon'(x) \leq (C_4^{1/2} + \rho_+) \varepsilon^{-1}
			\quad\text{for }
			x \in (1-\varepsilon, 1)
			\text{.}
		\end{aligned}
		\end{equation*}

		Now, let $\delta>0$ be a fixed small constant (independent of $\varepsilon$) 
		that satisfies
		\begin{equation*}
		\begin{aligned}
			\delta < 1\quad\text{and }
			\frac{\rho_+}{2} \geq \delta (C_4^{1/2} + \rho_+)
			\text{.}
		\end{aligned}
		\end{equation*}
		For $x\in (1-\delta\varepsilon, 1)$, we have, for some $x^*\in(x,1)$,
		\begin{equation}\label{ue2est}
		\begin{aligned}
			&w_\varepsilon(x)
			= \rho_+ - (1-x)w_\varepsilon'(x^*) 
			\geq \rho_+ -\delta\varepsilon \times (C_4^{1/2} + \rho_+)\varepsilon^{-1} 
			\geq \frac{\rho_+}{2}
			\text{,}
			\quad\text{or equivalently,}\\
			&u_\varepsilon''(x) \leq \frac{2}{\rho_+}
			\text{.}
		\end{aligned}
		\end{equation}
		Choose $n$ large so that $\varepsilon_n$ is small enough to satisfy
		\begin{equation}\label{delta1}
		\begin{aligned}
			\delta \varepsilon_n \frac{2}{\rho_+} \leq \frac{m}{2\varepsilon_n}
			\text{,}
		\end{aligned}
		\end{equation}
		\begin{equation}\label{delta2}
		\begin{aligned}
			\frac{2-\delta\varepsilon_n -\varepsilon_n^{1/4}}{2-\delta\varepsilon_n} \geq \frac{1}{2}
			\text{,}
		\end{aligned}
		\end{equation}
		\begin{equation}\label{delta3}
		\begin{aligned}
			1-\delta\varepsilon_n - \varepsilon_n^{1/4} > b
			\text{,}\quad\text{and }
		\end{aligned}
		\end{equation}
		\begin{equation}\label{delta4}
		\begin{aligned}
			(\delta\varepsilon_n + \varepsilon_n^{1/4})\norm{\varphi'}_{L^\infty(-1,1)} 
			\leq \frac{m\delta}{8}
			\text{.}
		\end{aligned}
		\end{equation}
		Considering (\ref{raahyp}), (\ref{ue2est}) and (\ref{delta1}), 
		we get for $x\in(1-\delta\varepsilon_n, 1)$,
		\begin{equation*}
		\begin{aligned}
			u_{\varepsilon_n}'(x) 
			&= u_{\varepsilon_n}'(1)-\int_x^1 u_{\varepsilon_n}''(t) \, dt
			&\geq \frac{m}{\varepsilon_n} - \delta{\varepsilon_n} \times \frac{2}{\rho_+}\\
			&\geq \frac{m}{2\varepsilon_n}
			\text{.}
		\end{aligned}
		\end{equation*}
		Thus, from $u_{\varepsilon_n}(1)=0$,
		\begin{equation}\label{uenear1}
		\begin{aligned}
			u_{\varepsilon_n}(1-\delta\varepsilon_n)
			\leq -\frac{m}{2\varepsilon_n}  \delta\varepsilon_n = -\frac{m\delta}{2}
			\text{.}
		\end{aligned}
		\end{equation}
		
		We now use the convexity of $u_{\varepsilon_n}$, (\ref{delta2})--(\ref{uenear1})
		and $u_{\varepsilon_n}(\pm 1) = 0$ to estimate
		$\norm{u_{\varepsilon_n}-\varphi}_{L^2(b,1)}$ to get a contradiction.
		For $x \in (1-\delta\varepsilon_n - \varepsilon_n^{1/4}, 1-\delta\varepsilon_n )$, 
		we define (see Figure \ref{uegraph})
		\begin{equation*}
		\begin{aligned}
			&A = (-1,0)\text{,}\quad
			B = (x,0)\text{,}\quad
			C = (1-\delta\varepsilon_n, 0)\text{,}\quad\\
			&D = (1-\delta\varepsilon_n, u_{\varepsilon_n}(1-\delta\varepsilon_n))\text{,}\quad
			E = (x, u_{\varepsilon_n}(x))\text{,}\quad\text{and }
			F = BE \cap AD \text{.}
		\end{aligned}
		\end{equation*}
		Because $u_{\varepsilon_n}$ is convex, its graph is below $AD$
		and therefore, $F$ is on the line segment $BE$.
		As the triangles $ABF$ and $ACD$ are similar, we have
		\begin{equation}\label{fig-1}
		\begin{aligned}
			\frac{BF}{CD} = \frac{AB}{AC}
			\text{.}
		\end{aligned}
		\end{equation}
		We also know that
		\begin{equation}
		\begin{aligned}
			BF < BE = |u_{\varepsilon_n} (x)|
			\text{,}\quad
			AC = 2- \delta\varepsilon_n
			\text{,}\quad
			AB = x+ 1\geq 2-\delta\varepsilon_n - \varepsilon_n^{1/4}
			\text{,}
		\end{aligned}
		\end{equation}
		and by (\ref{uenear1}), 
		\begin{align}\label{fig-1-0}
			CD = |u_{\varepsilon_n}(1-\delta\varepsilon_n)|
			\geq \frac{m\delta}{2}
			\text{.}
		\end{align}
		\begin{figure}[ht]
			\centering
			\begin{tikzpicture}[line cap=round,line join=round,>=triangle 45,x=3cm,y=3cm]
				\begin{axis}[
					x=3cm,y=3cm,
					axis x line=middle,
					axis y line=none,
					xmin=-1.4,
					xmax=1.4,
					ymin=-1.7,
					ymax=0.2,
					xticklabel=\empty,yticklabel=\empty,]
					\draw[line width=0.5pt,color=black,smooth,samples=100,domain=-1:1] 
						plot(\x,{5/2*(((\x)/2+1/2)^(4)-((\x)/2+1/2))});
					\draw[line width=0.5pt] (-1,0)-- (0.15835300626052404,-1.1666321803457715);
					\draw[line width=0.5pt] (-0.30030884186063633,0)-- (-0.30030884186063633,-0.837164486844264);
					\draw[line width=0.5pt] (0.15835300626052404,0)-- (0.15835300626052404,-1.1666321803457715);
					\begin{scriptsize}
						\draw[color=black] (1.0596640448170747,0.1) node {};
						\draw[color=black] (-0.9816315546023938,0.1) node {$A$};
						\draw[color=black] (0.16,-1.27) node {$D$};
						\draw[color=black] (-0.32,-0.97) node {$E$};
						\draw[color=black] (0.16,0.1) node {$C$};
						\draw[color=black] (-0.27907182255287494,0.1) node {$B$};
						\draw[color=black] (-0.22,-0.6504352621730012) node {$F$};
						\draw[color=black] (-0.22,-0.1) node {$x$};
						\draw[color=black] (0.4,-0.1) node {$1-\delta\varepsilon_n$};
						\draw[color=black] (0.75, -1.0) node {$u_{\varepsilon_n}$};
						\draw[color=black] (-1.07,-0.1) node {$-1$};
						\draw[color=black] (1.05,-0.1) node {$1$};
					\end{scriptsize}
				\end{axis}
			\end{tikzpicture}
			\caption{Construction of the points $A$--$F$}\label{uegraph}
		\end{figure}
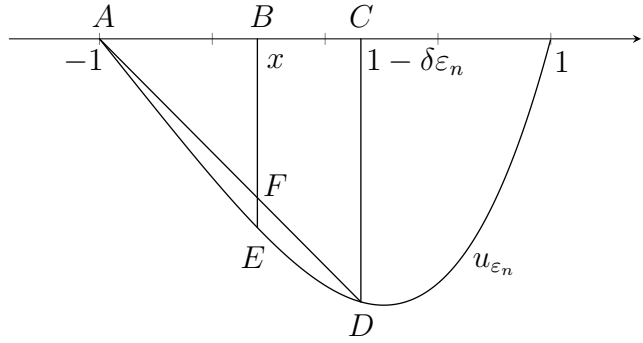

		Therefore, from (\ref{delta2}) and (\ref{fig-1})--(\ref{fig-1-0}), we have
		\begin{equation}\label{dist1}
		\begin{aligned}
			|u_{\varepsilon_n} (x) |
			&> BF =  \frac{AB}{AC} CD 
			\geq \frac{(2-\delta\varepsilon_n - \varepsilon_n^{1/4})}{2-\delta\varepsilon_n} \frac{m\delta}{2} \\
			&\geq \frac{1}{2} \frac{m\delta}{2} = \frac{m\delta}{4}
			\text{.}
		\end{aligned}
		\end{equation}

		Also, from (\ref{delta4}) and $\varphi(1) = 0$, for 
		$x\in(1-\delta\varepsilon_n-\varepsilon_n^{1/4}, 1-\delta\varepsilon_n)$
		we get
		\begin{equation}\label{dist2}
		\begin{aligned}
			|\varphi(x)| 
			&\leq\int_x^1 |\varphi'(t)| \,dt 
			\leq (1-x) \norm{\varphi'}_{L^\infty(-1,1)} \\
			&\leq (\delta\varepsilon_n+\varepsilon_n^{1/4})\norm{\varphi'}_{L^\infty(-1,1)}\\
			&\leq\frac{m\delta}{8}
			\text{.}
		\end{aligned}
		\end{equation}
		Putting (\ref{dist1}) and (\ref{dist2}) together yields
		\begin{equation*}
		\begin{aligned}
			|u_{\varepsilon_n} - \varphi| \geq \frac{m\delta}{8}
			\quad\text{in }
			(1-\delta\varepsilon_n-\varepsilon_n^{1/4}, 1-\delta\varepsilon_n)
			\text{.}
		\end{aligned}
		\end{equation*}
		
		Therefore, we conclude from (\ref{delta3}) that
		\begin{equation}\label{1dl1bd}
		\begin{aligned}
			\norm{u_{\varepsilon_n}-\varphi}_{L^2(b, 1)}^2
			\geq \int_{1-\delta\varepsilon_n -\varepsilon_n^{1/4}}^{1-\delta\varepsilon_n} |u_{\varepsilon_n}(x)-\varphi(x)|^2\, dx 
			\geq \left(\frac{m\delta}{8}\right)^2 \varepsilon_n^{1/4}
			\text{.}
		\end{aligned}
		\end{equation}
		However, (\ref{1-ide6}) gives the bound
		\begin{equation*}
		\begin{aligned}
			\norm{u_{\varepsilon_n} -\varphi}_{L^2(b,1)} \leq C_4^{1/2} \varepsilon_n^{1/2}
			\text{,}
		\end{aligned}
		\end{equation*}
		which contradicts (\ref{1dl1bd}) for small values of $\varepsilon_n$.
		This completes the proof of the Lemma.
		\end{proof}

		Now, we prove Theorem \ref{mainthm}\ref{thm:conv}.

		\begin{proof}[Proof of Theorem \ref{mainthm}\ref{thm:conv}]
		By (\ref{1-ide8}), the family $(u_\varepsilon)$ of 
		$W^{4,\infty}(-1,1)$ solutions to (\ref{1deqn}) satisfies
		\begin{equation}\label{unifbdd}
		\begin{aligned}
			\norm{u_\varepsilon}_{L^\infty(-1,1)} \leq C
		\end{aligned}
		\end{equation}
		for $C$ independent of $\varepsilon$.
		Furthermore, for any interval $I=[t_1,t_2]$ compactly supported in $(-1,1)$,
		we can combine (\ref{unifbdd}) with the gradient bound (\ref{1-ide2-1}) to obtain
		\begin{equation}\label{unifgradbdd}
		\begin{aligned}
			\norm{u_\varepsilon'}_{L^\infty(I)} \leq \widetilde{C}(t_1, t_2) 
			= \widetilde{C}(I)
			\text{.}
		\end{aligned}
		\end{equation}
		Here, $\widetilde{C}$ does not depend on $\varepsilon$ but on the distance of the set $I$
		to the exterior of $(-1,1)$.
		From (\ref{unifbdd}) and (\ref{unifgradbdd}), by
		passing to a subsequence $\varepsilon_k\rightarrow 0$, we have 
		\begin{equation}\label{1dsubconv}
		\begin{aligned}
			&u_{\varepsilon_k}\rightarrow u 
			\quad\text{weakly in } W^{1,2}(a,b)\text{,}\quad\text{and} \\
			&u_{\varepsilon_k}\rightarrow u 
			\quad\text{uniformly on compact intervals in }
			(-1,1)\text{,}
		\end{aligned}
		\end{equation}
		for some convex function $u$ in $(-1,1)$.
		From (\ref{1-ide6}), we have
		\begin{equation}\label{1dl2bd}
		\begin{aligned}
			\int_{(-1,a)\cup(b,1)} (u_{\varepsilon_k}-\varphi)^2\,dx
			\leq C_4\varepsilon_k 
			\rightarrow 0
			\quad\text{as }k\rightarrow 0
			\text{.}
		\end{aligned}
		\end{equation}
		Therefore, from (\ref{1dsubconv}), we have $u=\varphi$ outside $(a,b)$
		and hence $u\in\overline{S}[\varphi]$.
		As in Le \cite{Singular_Abreu, Convex_Approx}
		we will prove that $u$ minimizes the functional $J$ given by (\ref{1dfnctl})
		over $\overline{S}[\varphi]$ defined by (\ref{1dvarprob}) 
		by the following steps.

		\emph{Step 1}.
		We show that 
		\begin{equation}\label{step1}
		\begin{aligned}
			\liminf_{k\rightarrow\infty} J(u_{\varepsilon_k}) \geq J(u)
			\text{.}
		\end{aligned}
		\end{equation}
		From the convexity of $F^0$ in $z$ and $F^1$ in $p$, we have
		\begin{equation*}
		\begin{aligned}
			&J(u_{\varepsilon_k}) - J(u) \\
			&= \int_a^b \left[F^0(x, u_{\varepsilon_k}(x)) - F^0(x, u(x)) \right] dx
			+ \int_a^b \left[ F^1 (x,u_{\varepsilon_k}'(x)) - F^1(x, u'(x)) \right] dx \\
			&\geq \int_a^b F_z^0 (x, u(x)) (u_{\varepsilon_k} - u) \, dx 
			+ \int_a^b F_p^1(x, u'(x)) (u_{\varepsilon_k}'(x) -u'(x)) \, dx
			\text{.}
		\end{aligned}
		\end{equation*}
		By (\ref{1dsubconv}) and $|F_z^0(x,u(x))|\leq \eta(C_3)$, 
		the right-hand side converges to $0$ as $k\rightarrow\infty$,
		and the desired inequality (\ref{step1}) holds.

		\emph{Step 2}.
		Suppose $v\in\overline{S}[\varphi]$ is given by $v=v_1+v_2$,
		where $v_1$ is convex and $v_2\in C^2([-1,1])$ satisfies
		$v_2''\geq\alpha > 0$.
		We show that
		\begin{equation}\label{step2}
		\begin{aligned}
			J(v)\geq J(u_{\varepsilon})-A(\varepsilon)
			\text{,}\quad\text{where }
			A(\varepsilon)\rightarrow 0
			\quad\text{as }\varepsilon\rightarrow 0
			\text{.}
		\end{aligned}
		\end{equation}

		We approximate $v$ by smooth functions using mollifiers.
		Let $\rho \geq 0$ be smooth, supported on $(-1,1)$, 
		and satisfy $\int_{-1}^1 \rho\, dx = 1$.
		Extend $\varphi$ to be $C^3$ and uniformly convex on a neighborhood of $[-1,1]$,
		and also extend $v$ by setting $v=\varphi$ outside $[-1,1]$.
		For $\delta>0$ sufficiently small define $v_\delta=v*\rho_\delta$,
		where $\rho_\delta(x) = \delta^{-1}\rho(\delta^{-1}x)$.
		Then, as $\delta\rightarrow 0$, we have
		\begin{equation}\label{vconv}
		\begin{aligned}
			&v_\delta\rightarrow v\quad\text{in }[-1,1]
			\text{,}\quad
			v_\delta'\rightarrow v'\quad\text{a.e. in }[-1,1]
			\text{,}\quad\text{and}\\
			&v_\delta^{(k)}\rightarrow v^{(k)}=\varphi^{(k)}
			\quad\text{in } [-1,a)\cup (b,1]
			\quad\text{for }k\leq 2\text{.}
		\end{aligned}
		\end{equation}

		Recall from (\ref{jedef}) that for $w\in C^2(-1,1)$,
		\begin{equation}\label{1d-2-0}
		\begin{aligned}
			J_\varepsilon (w)
			= J(w) - \varepsilon\int_{-1}^1 \log w''\,dx
			+\frac{1}{2\varepsilon} \int_{(-1,a)\cup(b,1)} (w-\varphi)^2\,dx
			\text{.}
		\end{aligned}
		\end{equation}
		From the convexity of $F^0$ and $F^1$,
		\begin{equation}\label{1d-2-1}
		\begin{aligned}
			&J(v_\delta) - J(u_\varepsilon)\\
			&= \int_a^b \left[ F^0(x, v_\delta(x)) - F^0(x, u_\varepsilon(x)) \right] dx
			+ \int_a^b \left[F^1 (x,v_\delta'(x)) - F^1(x, u_\varepsilon'(x)) \right] dx \\
			&\geq \int_a^b F_z^0 (x, u_\varepsilon(x)) (v_\delta - u_\varepsilon) \, dx 
			+ \int_a^b F_p^1(x, u_\varepsilon'(x)) (v_\delta'(x) -u_\varepsilon'(x)) \, dx\\
			&= [F_p^1 (x, u_\varepsilon') (v_\delta-u_\varepsilon)]_a^b
			+ \int_a^b \left[F_z^0 (x, u_\varepsilon(x)) - 
				\frac{\partial}{\partial x} 
				\left(F_p^1(x, u_\varepsilon'(x))\right) \right] (v_\delta-u_\varepsilon) \, dx \\
			&= [F_p^1 (x, u_\varepsilon') (v_\delta-u_\varepsilon)]_a^b
			+ \int_a^b \varepsilon w_\varepsilon'' (v_\delta-u_\varepsilon) \,dx
			\text{.}
		\end{aligned}
		\end{equation}
		As $x \mapsto x^2$ is convex, we have
		\begin{equation}\label{1d-2-2}
		\begin{aligned}
			\frac{1}{2\varepsilon} \int_{(-1,a)\cup (b,1)} (v_\delta-\varphi)^2\,dx
			&- \frac{1}{2\varepsilon} \int_{(-1,a)\cup (b,1)} (u_\varepsilon-\varphi)^2\,dx \\
			&\geq \frac{1}{\varepsilon}\int_{(-1,a)\cup(b,1)}(u_\varepsilon-\varphi)(v_\delta-u_\varepsilon)\,dx\\
			&= \int_{(-1,a)\cup (b,1)}\varepsilon w_\varepsilon'' (v_\delta-u_\varepsilon)\,dx
			\text{.}
		\end{aligned}
		\end{equation}
		As $x \mapsto \log x$ is concave, we have
		\begin{equation}\label{1d-2-3}
		\begin{aligned}
			\varepsilon\int_{-1}^1 \log u_\varepsilon''\,dx 
			&- \varepsilon\int_{-1}^1 \log v_\delta''\,dx\\
			&\geq \varepsilon\int_{-1}^1 \frac{1}{u_\varepsilon''} (u_\varepsilon'' - v_\delta'')\,dx
			= \varepsilon\int_{-1}^1 w_\varepsilon (u_\varepsilon''-v_\delta'')\,dx \\
			&= \varepsilon\left([w_\varepsilon (u_\varepsilon'-v_\delta')]_{-1}^1
				- [w_\varepsilon' (u_\varepsilon-v_\delta)]_{-1}^1
				+ \int_{-1}^1 w_\varepsilon''(u_\varepsilon-v_\delta)\,dx \right)
			\text{.}
		\end{aligned}
		\end{equation}
		
		We also have
		\begin{equation}\label{1-e1-1}
		\begin{aligned}
			&\frac{1}{2\varepsilon}\int_{(-1,a)\cup(b,1)} \left[(u_\varepsilon-\varphi)^2-(v_\delta-\varphi)^2 \right] dx
			\geq -\frac{1}{2\varepsilon}\int_{(-1,a)\cup(b,1)} (v_\delta-\varphi)^2  dx
			\text{,}
			\quad\text{and }\\
			&- \varepsilon\int_{-1}^1 \log u_\varepsilon''\, dx
			\geq  - \varepsilon\int_{-1}^1 u_\varepsilon''\, dx 
			= - \varepsilon(u_\varepsilon'(1)-u_\varepsilon'(-1))
			\text{.}
		\end{aligned}
		\end{equation}
		Therefore, from (\ref{1d-2-1})--(\ref{1-e1-1}), we have
		\begin{equation}\label{2-ide1}
		\begin{aligned}
			J(v_\delta)-J(u_\varepsilon)
			&\geq  [F_p^1 (x, u_\varepsilon') (v_\delta-u_\varepsilon)]_a^b
			+ [\varepsilon w_\varepsilon (u_\varepsilon' - v_\delta')]_{-1}^1
			- [\varepsilon w_\varepsilon' (u_\varepsilon-v_\delta)]_{-1}^1 \\
			&-\frac{1}{2\varepsilon}\int_{(-1,a)\cup(b,1)} (v_\delta-\varphi)^2  dx
			+ \varepsilon\int_{-1}^1 \log v_\delta''\, dx - \varepsilon(u_\varepsilon'(1)-u_\varepsilon'(-1))
			\text{.}
		\end{aligned}
		\end{equation}

		We first let $\delta\rightarrow 0$ in (\ref{2-ide1}).
		From (\ref{vconv}), we get
		\begin{equation}
		\begin{aligned}
			[F_p^1 (x, u_\varepsilon') (v_\delta-u_\varepsilon)]_a^b
			&\rightarrow [F_p^1 (x, u_\varepsilon') (v-u_\varepsilon)]_a^b
			\text{,} \quad\text{and }\\
			[\varepsilon w_\varepsilon (u_\varepsilon' - v_\delta')]_{-1}^1
			&\rightarrow [\varepsilon w_\varepsilon (u_\varepsilon' - v')]_{-1}^1
			\text{.}
		\end{aligned}
		\end{equation}
		As $v=\varphi$ outside $(a,b)$ and $u_\varepsilon(\pm 1)=\varphi(\pm 1)=0$, 
		from (\ref{vconv}) we also have
		\begin{equation}
		\begin{aligned}
			-\frac{1}{2\varepsilon}\int_{(-1,a)\cup(b,1)} (v_\delta-\varphi)^2  dx
			\rightarrow 0
			\text{,}\quad\text{and } 
			[\varepsilon w_\varepsilon' (u_\varepsilon-v_\delta)]_{-1}^1 
			\rightarrow 0
			\quad\text{as }\delta\rightarrow 0
			\text{.}
		\end{aligned}
		\end{equation}

		Recall that $v_\delta=v*\rho_\delta$, and $v=v_1+v_2$ where $v_2''\geq\alpha$.
		Hence, $v_\delta'' \geq \alpha$, which yields
		\begin{equation}\label{2-ide1-0}
		\begin{aligned}
			\liminf_{\delta\rightarrow 0} 
			\left(\varepsilon\int_{-1}^1 \log v_\delta'' \, dx \right) 
			\geq 2\varepsilon\log \alpha
			\text{.}
		\end{aligned}
		\end{equation}
		By (\ref{vconv}), $J(v_\delta)\rightarrow J(v)$ as $\delta\rightarrow 0$.
		Considering (\ref{2-ide1})--(\ref{2-ide1-0}), we get
		\begin{equation}\label{2-ide2}
		\begin{aligned}
			J(v)&-J(u_\varepsilon)\\
			&\geq [F_p^1 (x, u_\varepsilon') (v-u_\varepsilon)]_a^b
			+  [\varepsilon w_\varepsilon (u_\varepsilon' - v')]_{-1}^1
			+2\varepsilon\log\alpha-\varepsilon(u_\varepsilon'(1)-u_\varepsilon'(-1))
			\text{.}
		\end{aligned}
		\end{equation}

		Now, we let $\varepsilon\rightarrow 0$ in (\ref{2-ide2}).
		First, we have
		\begin{equation}
		\begin{aligned}
			\relax[F_p^1 (x, u_\varepsilon')(v-u_\varepsilon)]_a^b 
			\rightarrow 0 \quad\text{as }\varepsilon \rightarrow 0
		\end{aligned}
		\end{equation}
		because $|F_p^1 (x, u_\varepsilon')| \leq \eta(D_1)$,		
		and $u_\varepsilon(t)$ converges to 
		$\varphi(t)=v(t)$ as $\varepsilon\rightarrow 0$ if $t\notin (a,b)$.
		By Lemma \ref{gradconvlemma}, we have
		\begin{equation}\label{1-e2}
		\begin{aligned}
			\varepsilon w_\varepsilon(\pm 1) u_\varepsilon'(\pm 1) = \rho_\pm \varepsilon u_\varepsilon'(\pm 1)
			\rightarrow 0
			\text{,}\quad\text{and }
			\varepsilon u_\varepsilon'(\pm 1)\rightarrow 0
			\quad\text{as }\varepsilon\rightarrow 0
		\end{aligned}
		\end{equation}
		We also have
		\begin{equation}\label{2-ide2-0}
		\begin{aligned}
			2\varepsilon\log\alpha\rightarrow 0
			\text{,}\quad\text{and}\quad
			\varepsilon w_\varepsilon(\pm 1) v'(\pm 1) = \varepsilon \rho_\pm v'(\pm 1)
			\rightarrow 0 \quad\text{as } \varepsilon\rightarrow 0
			\text{.}
		\end{aligned}
		\end{equation}
		Putting (\ref{2-ide2})--(\ref{2-ide2-0}) together completes the proof of Step 2.

		\begin{remark}\label{rmk:gradconv}
		The term $\varepsilon w_\varepsilon(\pm 1) u_{\varepsilon}'(\pm 1)$
		corresponds to $\varepsilon^{(n-1)/n}\eta_\varepsilon$ 
		from the proof of Le \cite[(3.20)]{Convex_Approx}.
		$\varepsilon^{(n-1)/n}$ converges to $0$ as $\varepsilon\rightarrow 0$ if $n\geq 2$,
		but this term is a constant when $n=1$
		and the estimate does not directly imply the result in our case.
		\end{remark}
		
		\emph{Step 3}.
		Finally, we show that $J(v)\geq J(u)$ for any $v\in\overline{S}[\varphi]$.

		Since $v\in\overline{S}[\varphi]$,
		$v_\lambda:=\lambda v + (1-\lambda)\varphi$ is in $\overline{S}[\varphi]$.
		Also, $v_\lambda = v_1 + v_2$ for $v_1=\lambda v$ and $v_2=(1-\lambda)\varphi$.
		Recalling that $\varphi''\geq c_0 >0$, 
		we find that $v$ satisfies the assumptions in Step 2.
		Therefore, from (\ref{step1}) and (\ref{step2}), we get
		\begin{equation}\label{unif1}
		\begin{aligned}
			J(v_\lambda)
			\geq \liminf_{k\rightarrow\infty} (J(u_{\varepsilon_k})-A(\varepsilon_k)) 
			= \liminf_{k\rightarrow\infty} J(u_{\varepsilon_k}) 
			\geq J(u)
			\quad\text{for all }\lambda\in(0,1)
			\text{.}
		\end{aligned}
		\end{equation}
		By definition, $J(v_\lambda)\rightarrow J(v)$ as $\lambda\rightarrow 1$.
		Therefore, passing to the limit of $\lambda\rightarrow 1$ in (\ref{unif1}),
		we conclude that $J(v)\geq J(u)$.
		Hence $u$ is a minimizer to the variational problem (\ref{1dfnctl})--(\ref{1dvarprob}).
		This completes the proof of Theorem \ref{mainthm}\ref{thm:conv}.
	\end{proof}

	\section{Characterization of Limiting Minimizers}\label{sec:eleqthm}
	In this section, we prove Theorem \ref{mainthm}\ref{thm:eqn} 
	by establishing (\ref{eleqdist}).
	
	\begin{proof}[Proof of Theorem \ref{mainthm}\ref{thm:eqn}]
	We start with the subsequence $(u_{\varepsilon_k})_{k}$ in (\ref{1dsubconv}).
	By the convexity of $u_{\varepsilon_k}$ we have, for any $x\in(a,b)$ and small $\delta > 0$,
	\begin{equation*}
	\begin{aligned}
		u_{\varepsilon_k}'(x) \leq \frac{u_{\varepsilon_k} (x+\delta)-u_{\varepsilon_k}(x)}{\delta}
		\text{.}
	\end{aligned}
	\end{equation*}
	As $u_{\varepsilon_k}$ converges uniformly on compact sets of $(-1,1)$ to the convex function $u$,
	\begin{equation*}
	\begin{aligned}
		\limsup_{k\rightarrow\infty} u_{\varepsilon_k}'(x)
		\leq \frac{u(x+\delta)-u(x)}{\delta}
	\end{aligned}
	\end{equation*}
	for all $x\in(a,b)$ and small $\delta > 0$. 
	Letting $\delta\rightarrow 0$, we can conclude that
	\begin{equation}\label{limsup}
	\begin{aligned}
		\limsup_{k\rightarrow\infty} u_{\varepsilon_k}'(x)
		\leq u'(x)
		\quad\text{for }x\in S
		\text{,}
	\end{aligned}
	\end{equation}
	where $S$ is the set of points of differentiability of $u$ in $(a,b)$.
	Using the same argument on
	$(u_{\varepsilon_k}(x)-u_{\varepsilon_k}(x-\delta))/\delta$,
	we obtain
	\begin{equation}\label{liminf}
	\begin{aligned}
		\liminf_{k\rightarrow\infty} u_{\varepsilon_k}'(x)\geq u'(x)
		\quad\text{for }x\in S
		\text{.}
	\end{aligned}
	\end{equation}
	By (\ref{limsup}) and (\ref{liminf}), $u_{\varepsilon_k}'$ 
	converges pointwise to $u'$ on $S$.
	The function $u$, being convex on the interval $(-1,1)$, is Lipschitz on $(a,b)$.
	Hence, $u$ is differentiable a.e. on $(a,b)$ by Rademacher's theorem. 
	Therefore, $u_{\varepsilon_k}'$ converges a.e. on $(a,b)$ to $u'$.

	Now we turn our attention to the equation
	\begin{equation}\label{sec3eqn}
	\begin{aligned}
		(\varepsilon w_\varepsilon)'' 
		&= \varepsilon w_\varepsilon''
		&=F_z^0 (x,u_\varepsilon) - (F_p^1 (x,u_\varepsilon'))'
		\quad\text{in }(a,b)
		\text{,}
	\end{aligned}
	\end{equation}
	and derive (\ref{eleqdist}) by showing the convergence of the two terms separately.
	First, from (\ref{sec3growth}), the uniform bound on $u_\varepsilon$,
	and the uniform convergence of $u_{\varepsilon_k}$ to $u$, we have 
	\begin{equation}\label{sec3conv1}
	\begin{aligned}
		F_z^0 (x,u_{\varepsilon_k}) \rightarrow F_z^0 (x,u)
		\quad\text{uniformly in }(a,b)
		\text{.}
	\end{aligned}
	\end{equation}
	Next, from the estimate (\ref{west}),
	$\varepsilon_k w_{\varepsilon_k} = {\varepsilon_k}/{u_{\varepsilon_k}''}$ 
	is uniformly bounded and hence has a subsequence $\varepsilon_{k_j} w_{\varepsilon_{k_j}}$
	converging weakly in $L^q(a,b)$ to a function $w\in L^q(a,b)$.
	Hence, we have
	\begin{equation}\label{sec3conv2}
	\begin{aligned}
		(\varepsilon_{k_j} w_{\varepsilon_{k_j}})'' \rightarrow w''
		\quad\text{in the sense of distributions.}
	\end{aligned}
	\end{equation}
	
	Finally, from (\ref{1dgrad}),
	$|F_p^1(x, u_{\varepsilon_{k_j}}')|\leq \eta_1(D_1)$ 
	and therefore $F_p^1(x, u_{\varepsilon_{k_j}}')$ is uniformly bounded in $j$.
	This, together with the
	almost everywhere convergence of $u_{\varepsilon_{k_j}}'$ to $u'$, the continuity of $F_p^1$,
	and the Dominated Convergence Theorem implies that
	$F_p^1(x,u_{\varepsilon_{k_j}}')$ converges to $F_p^1(x,u')$
	strongly in $L^r(a,b)$ for $1\leq r < \infty$.
	This yields
	\begin{equation}\label{sec3conv3}
	\begin{aligned}
		(F_p^1(x,u_{\varepsilon_{k_j}}'))' 
		\rightarrow (F_p^1(x,u'))'
		\quad\text{in the sense of distributions.}
	\end{aligned}
	\end{equation}
	
	Hence, passing to the limit along the subsequence $(u_{\varepsilon_k})$ 
	in (\ref{sec3eqn}) and
	applying (\ref{sec3conv1})--(\ref{sec3conv3}), we obtain
	\begin{equation}
	\begin{aligned}
		w'' = F_z^0(x,u) - (F_p^1(x,u'))'
		\quad\text{in the sense of distributions.}
	\end{aligned}
	\end{equation}
	This gives us (\ref{eleqdist}) as asserted.
	The proof of the Theorem is complete.
	\end{proof}

	\section{Conclusion}\label{sec:conclusion}

	The primary focus of this note is the study of fourth-order Abreu-type equations in dimension one.
	In dimensions higher than or equal to two, 
	employing an approximation scheme introduced by Carlier-Radice \cite{CR}
	and extended by Le \cite{Singular_Abreu} has enabled authors to use
	solutions to the second boundary value problem for Abreu-type equations 
	to approximate minimizers of convex functionals 
	subject to convexity constraint in the form of (\ref{hdvar})--(\ref{hdvarcon})
	\cite{CR,Singular_Abreu,Convex_Approx,Twisted_Harnack,Abreu_HD}.

	In dimension one, 
	Abreu-type equations can exhibit various solvability phenomena, 
	as highlighted in Chau-Weinkove \cite[Proposition 3.2]{CW}.
	In this note,
	we have demonstrated that for Abreu-type equations with a singular term, 
	solvability results similar to those in higher dimensions are achievable;
	see Theorem \ref{mainthm}\ref{thm:est}--\ref{thm:conv}.
	Additionally, we obtained a new estimate in (\ref{west}) for solutions.
	By combining this estimate with the approximation scheme in 
	Theorem \ref{mainthm}\ref{thm:est}--\ref{thm:conv},
	we have characterized limiting minimizers as stated in Theorem \ref{mainthm}\ref{thm:eqn}.

	Future research could explore the following issues:
	\begin{enumerate}
		\item Since our characterization in Theorem \ref{mainthm}\ref{thm:eqn}
			relies on the approximation scheme,
			it only applies to minimizers that can be approximated as limits of solutions to (1.9).
			Therefore, it would be interesting to find out whether there are minimizers 
			that cannot be approximated in this matter,
			and if so, determine if similar characterization can be achieved for these minimizers.
			Also see Remark \ref{rmk:Lions}.
		\item While the 
			approximation scheme in Theorem \ref{mainthm}\ref{thm:est}--\ref{thm:conv} is already 
			established in higher dimensions,
			estimates similar to (\ref{west}) have not been proved, 
			to the best of the author's knowledge. 
			This would be the missing part in obtaining characterizations 
			similar to Theorem \ref{mainthm}\ref{thm:eqn} in higher dimensions.
			Determining whether such estimates hold for solutions to Abreu-type equations 
			in dimensions at least two could be a direction for further study.
	\end{enumerate}

	\textbf{Acknowledgements.}
	The author would like to thank his advisor, Professor N.Q. Le,
	for suggesting the problem,
	and providing helpful guidance and advice throughout the course of this work.

	The author would also like to thank the anonymous referee for providing constructive feedback,
	which helped the author in improving this note.

	The research of the author was supported in part by NSF grant DMS-2054686.

\end{document}